\documentclass[reqno]{amsart}
\usepackage[english]{babel}
\usepackage{enumerate} 
\usepackage[reqno]{amsmath}\usepackage{amsfonts}
\usepackage{amssymb} 
\usepackage{hyperref} 
\usepackage{secdot} 
\usepackage{amsthm} 
\usepackage{indentfirst} 
\newtheorem{theorem}{Theorem}
\newtheorem{proposition}[theorem]{Proposition}
\newtheorem{lemma}[theorem]{Lemma}
\newtheorem{corollary}[theorem]{Corollary}

\theoremstyle{definition}
\newtheorem{df}[theorem]{Definition}
\newtheorem{example}[theorem]{Example}
\newtheorem{notation}[theorem]{Notation}

\theoremstyle{remark}
\newtheorem{remark}[theorem]{Remark}
\numberwithin{equation}{section} 
\DeclareMathOperator{\Supp}{Supp} \DeclareMathOperator{\CD}{CD} \DeclareMathOperator{\chara}{char} \DeclareMathOperator{\lspan}{span}
\begin{document}

\begin{center}
\textbf{\Large GRADINGS ON COMPOSITION SUPERALGEBRAS}
\\[1cm]
Diego Aranda-Orna\footnote{Supported
by the Spanish
Ministerio de Econom\'{\i}a y Competitividad and FEDER (MTM 2010-18370-C04-02) and by
the Diputaci\'on General de Arag\'on y el Fondo Social Europeo (Grupo de Investigaci\'on de \'Algebra)}
\\[.3cm]
Departamento de Matem\'aticas, Facultad de Ciencias,\\
Universidad de Zaragoza, 50009 Zaragoza, Spain\\
E-mail: daranda@unizar.es
\\[1cm]
\end{center}

{\noindent \small ABSTRACT. We classify up to equivalence the gradings on Hurwitz superalgebras and on symmetric composition superalgebras, over any field. Also, classifications up to isomorphism are given in case the field is algebraically closed. By grading, here we mean group grading. }

\section{INTRODUCTION AND GENERALITIES}
Gradings on composition algebras are useful to find gradings on exceptional simple Lie algebras. Gradings on Hurwitz algebras and on symmetric composition algebras were described in \cite{Eld98} and \cite{Eld09}. The aim of this article is to extend the known results of gradings on composition algebras to composition superalgebras. Hurwitz superalgebras and symmetric composition superalgebras were classified in \cite{EldOku02}.

\smallskip

We begin by giving the basic definitions and notations of gradings in algebras and superalgebras. (See \cite[Chap. 1]{EldKoc13}.)

Recall that given a set $S$, an {\it $S$-grading} on an $F$-algebra $\mathcal A$ is a vector space decomposition $\Gamma:{\mathcal A}=\bigoplus_{s\in S}{\mathcal A}^s$ such that for each $s_1,s_2\in S$ there is $s\in S$ with ${\mathcal A}^{s_1}{\mathcal A}^{s_2}\subseteq {\mathcal A}^s$. The {\it support} of the grading $\Gamma$ is the set $\Supp\Gamma:=\{s\in S:{\mathcal A}^s\neq 0\}$.

Let $\Gamma:{\mathcal A}=\bigoplus_{s\in S}{\mathcal A}^s$ and $\Gamma':{\mathcal A}=\bigoplus_{t\in T}{\mathcal A}^t$ be gradings on $\mathcal A$. We say that $\Gamma$ is a {\it refinement} of $\Gamma'$, or that $\Gamma'$ is a \textit{coarsening} of $\Gamma$, if for each $s\in S$ there is some $t\in T$ such that ${\mathcal A}^s\subseteq {\mathcal A}^t$. A grading is said to be {\it fine} if it admits no proper refinement.

Given a group $G$, a {\it $G$-grading} on an $F$-algebra $\mathcal A$ is a vector space decomposition $\Gamma:{\mathcal A}=\bigoplus_{g\in G}{\mathcal A}^g$, where ${\mathcal A}^g{\mathcal A}^h\subseteq {\mathcal A}^{gh}$ for each $g,h\in G$. Then,
given an $S$-grading $\Gamma:{\mathcal A}=\bigoplus_{s\in S}{\mathcal A}^s$ by a set $S$, we define the {\it universal grading group} as the group $G$ generated by $\Supp\Gamma$ with the relations $s_1s_2=s_3$ when $0\neq {\mathcal A}^{s_1}{\mathcal A}^{s_2}\subseteq {\mathcal A}^{s_3}$ (it is clear that then $\Gamma$ induces a $G$-grading on $\mathcal A$). If the original grading is a grading by a group $H$, then there is a natural homomorphism $G\rightarrow H$ which is the identity on the support of the grading.

Given a $G$-grading $\Gamma:{\mathcal A}=\bigoplus_{g\in G}{\mathcal A}^g$ and a group homomorphism $\alpha:G\rightarrow H$, the $H$-grading $\Gamma':{\mathcal A}=\bigoplus_{h\in H}{\mathcal A}^h$, where ${\mathcal A}^h=\bigoplus_{g\in\alpha^{-1}(h)}{\mathcal A}^g$, is a coarsening of $\Gamma$. It is called the grading induced from $\Gamma$ by $\alpha$. Any grading by a group $H$ is induced from a fine (group) grading by a homomorphism from the universal grading group of the fine grading into $H$.

Given two gradings $\Gamma:{\mathcal A}=\bigoplus_{s\in S}{\mathcal A}^s$ and $\Gamma':{\mathcal B}=\bigoplus_{t\in T}{\mathcal B}^t$, an isomorphism $\varphi:{\mathcal A}\rightarrow{\mathcal B}$ of algebras is said to be an {\it equivalence of graded algebras} if for each $s\in S$ there is $t\in T$ such that $\varphi({\mathcal A}^s)={\mathcal B}^t$. Notice that the assignment $t:=\alpha(s)$ gives an isomorphism among the corresponding universal grading groups.

If the two gradings above are group gradings by the same group: $\Gamma:{\mathcal A}=\bigoplus_{g\in G}{\mathcal A}^g$ and $\Gamma':{\mathcal B}=\bigoplus_{g\in G}{\mathcal B}^g$ and the isomorphism above $\varphi$ satisfies $\varphi({\mathcal A}^g)={\mathcal B}^g$ for any $g\in G$, then $\varphi$ is called an \textit{isomorphism of $G$-graded algebras}.

\smallskip

A {\it superalgebra} is just a $\mathbb{Z}_2$-graded algebra. Given a superalgebra, we will write its $\mathbb{Z}_2$-grading with subindex notation, ${\mathcal A}={\mathcal A}_{\bar 0}\oplus{\mathcal A}_{\bar 1}$, and this will be called the \textit{main $\mathbb{Z}_2$-grading} (or just the {\it main grading}) in order to distinguish it from other possible $\mathbb{Z}_2$-gradings. The homogeneous components ${\mathcal A}_{\bar0}$ and ${\mathcal A}_{\bar1}$ are called the {\it even} and {\it odd components}, and their elements are called {\it even} and {\it odd} respectively. Note that an isomorphism of superalgebras is just an isomorphism of graded algebras.

Let $\Gamma : {\mathcal A} = \bigoplus_{g\in G} {\mathcal A}^g$ be a grading on the superalgebra ${\mathcal A}={\mathcal A}_{\bar 0}\oplus{\mathcal A}_{\bar 1}$ considered as an algebra. We will say that $\Gamma$ is a {\it grading on (the superalgebra)} $\mathcal A$ (or that $\Gamma$ is compatible with the main grading) if
$$
{\mathcal A}^g=({\mathcal A}^g \cap {\mathcal A}_{\bar 0}) \oplus ({\mathcal A}^g \cap {\mathcal A}_{\bar 1})
$$
for all $g\in G$. We use the notation ${\mathcal A}^g_i={\mathcal A}^g \cap {\mathcal A}_i$ for any $g\in G$, $i=\bar0,\bar1$. An equivalence (resp. isomorphism) of graded algebras which is also an isomorphism of superalgebras is called an {\it equivalence} (resp. {\it isomorphism}) of graded superalgebras.

The {\it trivial grading} on a superalgebra ${\mathcal A}$ is given by ${\mathcal A}={\mathcal A}^e$ ($G=\{e\}$). (Notice that other authors may refer to the main grading as the trivial grading, and always consider the homogeneous components contained in ${\mathcal A}_{\bar0}$ or ${\mathcal A}_{\bar1}$; but this is not our case).

From now on, only group gradings will be considered, unless otherwise stated. Therefore, the term fine grading will refer to a grading that admits no proper refinement in the class of group gradings.


\bigskip

We will recall now some well known facts about Cayley algebras, such as the existence of a ``canonical'' basis in the split Cayley algebra, and the Cayley-Dickson process. The kind of arguments used to show the existence of these bases will be explained in this section, since they will be used in several proofs of the following sections without explaining again all the details.

\smallskip

Recall that a Hurwitz algebra $C$ is a unital algebra over a field $F$ endowed with a regular (definition as in \cite[p. xix]{KMRT}) quadratic form $q:C\rightarrow F$ which is multiplicative: $q(xy)=q(x)q(y)$ for any $x,y\in C$. For the basic facts about Hurwitz algebras the reader may consult \cite[Chap. 2]{ZSSS} or \cite[Chap. 8]{KMRT}. The dimension of any Hurwitz algebra is restricted to $1$, $2$, $4$ or $8$. The $4$-dimensional Hurwitz algebras are the quaternion algebras, and the $8$-dimensional Hurwitz algebras are called Cayley algebras or octonion algebras.

Denote also by $q$ the polar form of the norm: $q(x,y)=q(x+y)-q(x)-q(y)$. Any element of a Hurwitz algebra satisfies the quadratic equation $x^2-q(x,1)x+q(x)1=0$, which can be written as $x\bar x=\bar x x=q(x)1$, where $\bar x=q(x,1)1-x$ (the conjugate of $x$). The map $x\mapsto \bar x$ is an involution and satisfies $q(xy,z)=q(y,\bar xz)=q(x,z\bar y)$ for any $x,y,z$.

A {\it split Hurwitz algebra} is a Hurwitz algebra $C$ with a nonzero isotropic element: $0\neq x\in C$ such that $q(x)=0$. Note that any Hurwitz algebra of dimension $\geq2$ over an algebraically closed field is split. Let $C$ be a split Cayley algebra and let $a$ be a nonzero isotropic element. In that case, we can take $b\in C$ such that $q(a,\bar b)=1$ ($q$ is regular). Let $e_1:=ab$. We have $q(e_1)=0$ and $q(e_1,1)=1$, so $e_1^2=e_1$. Let $e_2:=\bar{e_1}=1-e_1$, so $q(e_2)=0$, $e_2^2=e_2$, $e_1e_2=0=e_2e_1$ and $q(e_1,e_2)=q(e_1,1)=1$. Then $K=Fe_1\oplus Fe_2$ is a composition subalgebra of $C$.

For any $x\in K^\perp$, $xe_1+\overline{xe_1}=q(xe_1,1)1=q(x,\bar e_1)1=q(x,e_2)1=0$. Hence $xe_1=-\bar e_1 \bar x=e_2 x$, and we get $xe_1=e_2x$, $xe_2=e_1x$.
Also, $x=1x=e_1x+e_2x$, and $e_2(e_1x)=(1-e_1)(e_1x)=((1-e_1)e_1)x=0=e_1(e_2x)$. Therefore, $K^\perp=U\oplus V$, with
\begin{align*}
U=\{x\in C:e_1x=x=xe_2, e_2x=0=xe_1\}=(e_1C)e_2, \\
V=\{x\in C:e_2x=x=xe_1, e_1x=0=xe_2\}=(e_2C)e_1.
\end{align*}
For any $u\in U$, $q(u)=q(e_1u)=q(e_1)q(u)=0$, so $U$ and $V$ are isotropic subspaces of $C$. Since $q$ is regular, $U$ and $V$ are paired by $q$ and $\dim U=\dim V=3$. Take $u_1,u_2\in U$ and $v\in V$. Then,
\begin{align*}
q(u_1u_2,K)\subseteq & q(u_1,Ku_2)\subseteq q(U,U)=0, \\
q(u_1u_2,v)=q(u_1u_2,e_2v) & =-q(e_2u_2,u_1v)+q(u_1,e_2)q(u_2,v)=0.
\end{align*}
Hence $U^2$ is orthogonal to $K$ and $V$, so $U^2\subseteq V$. Also $V^2\subseteq U$. Besides,
\begin{align*}
q(U,UV)\subseteq q(U^2,V)\subseteq q(V,V)=0, \\
q(UV,V)\subseteq q(U,V^2)\subseteq q(U,U)=0,
\end{align*}
so $UV+VU\subseteq K$. Moreover, $q(UV,e_1)\subseteq q(U,e_1V)=0$, so that $UV\subseteq Fe_1$ and $VU\subseteq Fe_2$. More precisely, for $u\in U$ and $v\in V$, $q(uv,e_2)=-q(u,e_2v)=-q(u,v)$, so that $uv=-q(u,v)e_1$, and $vu=-q(u,v)e_2$. Then, the decomposition $C=K\oplus U\oplus V$ is a $\mathbb{Z}_3$-grading on $C$.

For linearly independent elements $u_1,u_2\in U$, take $v\in V$ with $q(u_1,v)\neq0=q(u_2,v)$. Then $(u_1u_2)v=-(u_1v)u_2=q(u_1,v)u_2\neq0$, and so $U^2\neq0$.
Moreover, the trilinear map 
\begin{align*}
U\times U\times U & \rightarrow F \\ 
(x,y,z) & \mapsto q(xy,z),
\end{align*}
is alternating (for any $x\in U$, $q(x)=0=q(x,1)$, so $x^2=0$ and hence $q(x^2,z)=0$; but $q(xy,y)=-q(x,y^2)=0$ too).

Take a basis $\{u_1,u_2,u_3\}$ of $U$ with $q(u_1u_2,u_3)=1$ (this is always possible because $q(U^2,U)\neq0$ since $q$ is regular). Then $\{v_1=u_2u_3, v_2=u_3u_1, v_3=u_1u_2\}$ is the dual basis in $V$ relative to $q$. We will say that $\{e_1,e_2,u_1,u_2,u_3,v_1,v_2,v_3\}$ is a {\it canonical basis} of the split Cayley algebra $C$, and its multiplication table is:

\begin{center}
\begin{tabular}{c||cc|ccc|ccc|}
 & $e_1$ & $e_2$ & $u_1$ & $u_2$ & $u_3$ & $v_1$ & $v_2$ & $v_3$ \\
 \hline\hline
 $e_1$ & $e_1$ & 0 & $u_1$ & $u_2$ & $u_3$ & 0 & 0 & 0 \\
 $e_2$ & 0 & $e_2$ & 0 & 0 & 0 & $v_1$ & $v_2$ & $v_3$ \\
 \hline
 $u_1$ & 0 & $u_1$ & 0 & $v_3$ & $-v_2$ & $-e_1$ & 0 & 0 \\
 $u_2$ & 0 & $u_2$ & $-v_3$ & 0 & $v_1$ & 0 & $-e_1$ & 0 \\
 $u_3$ & 0 & $u_3$ & $v_2$ & $-v_1$ & 0 & 0 & 0 & $-e_1$ \\
 \hline
 $v_1$ & $v_1$ & 0 & $-e_2$ & 0 & 0 & 0 & $u_3$ & $-u_2$ \\
 $v_2$ & $v_2$ & 0 & 0 & $-e_2$ & 0 & $-u_3$ & 0 & $u_1$ \\
 $v_3$ & $v_3$ & 0 & 0 & 0 & $-e_2$ & $u_2$ & $-u_1$ & 0 \\
 \hline
\end{tabular}
\end{center}

A canonical basis determines a $\mathbb{Z}^2$-grading on $C$, where $\deg(u_1)=(1,0)=-\deg(v_1)$ and $\deg(u_2)=(0,1)=-\deg(v_2)$, which is called the {\it Cartan grading} (see \cite{EldKoc12}). On the split quaternion algebra, we have the canonical basis given by $\{e_1, e_2, u_1, v_1\}$, with the multiplication as in the table above. On the $2$-dimensional split Hurwitz algebra, the canonical basis is given by the orthogonal idempotents $\{e_1,e_2\}$.

\begin{notation}\label{CDdef} If $C$ is a Hurwitz algebra, with norm $q$, obtained by the Cayley-Dickson doubling process (see \cite[Chap. 2]{ZSSS}), then there is a Hurwitz subalgebra $Q$ and an element $u\in C$ such that $q(u)=\alpha\neq0$ and $C=Q\oplus Qu$. The multiplication is then given by
\begin{equation}
(a+bu)(c+du)=(ab-\alpha\bar d b)+(da+b\bar c)u,
\end{equation}
for any $a,b,c,d\in Q$. We will write then $C=\CD(Q,\alpha)$. Then $C$ becomes a superalgebra with $C_{\bar 0}=Q$ and $C_{\bar 1}=Qu$. We will refer to this superalgebra as the superalgebra $\CD(Q,\alpha)$.
\end{notation}


\begin{remark}\label{groupisabelian} 
We will see that for all the $G$-gradings studied on this paper, when $G$ is generated by $\Supp\Gamma$, then $G$ is abelian. This is already known for Hurwitz algebras and symmetric composition algebras (see \cite{Eld98} and \cite{Eld09}), so this is true for the superalgebras defined on them (for instance, for the superalgebras $\CD(Q,\alpha)$), and this will be checked too for the other Hurwitz superalgebras and symmetric composition superalgebras. Hence, if $\Gamma$ is a $G$-grading on one of these types of superalgebras, its support $\Supp\Gamma$ generates an abelian subgroup of $G$. We may always assume that $G$ is generated by $\Supp\Gamma$. So, in the classifications up to equivalence or up to isomorphism of the gradings in this paper, only abelian groups will be considered.
\end{remark}

\bigskip

Classifications of the gradings up to equivalence over any field, and up to isomorphism over an algebraically closed field, are given for Hurwitz superalgebras in Section 2, and for symmetric composition superalgebras in Section 3.

\bigskip

This work is based on the author's Master Thesis, written under the supervision of Alberto Elduque. I would like to thank both A. Elduque and the referee, for their so many suggestions and corrections.

\bigskip

\section{GRADINGS ON HURWITZ SUPERALGEBRAS}

The goal of this section is to classify gradings on Hurwitz superalgebras, up to equivalence over any field, and up to isomorphism in the algebraically closed case. First of all, we recall some definitions from \cite{EldOku02}.

\begin{df}
A {\it quadratic superform} on a $\mathbb{Z}_2$-graded vector space (or {\it superspace}) $V=V_{\bar0}\oplus V_{\bar1}$ over a field $F$ is a pair $q=(q_{\bar0},b)$ where
\begin{enumerate}
\item[i)] $q_{\bar0}:V_{\bar0}\rightarrow F$ is a usual quadratic form,
\item[ii)] $b:V\times V\rightarrow F$ is a supersymmetric even bilinear form. That is, $b|_{V_{\bar0}\times V_{\bar0}}$ is symmetric, $b|_{V_{\bar1}\times V_{\bar1}}$ is alternating and $b(V_{\bar0},V_{\bar1})=b(V_{\bar1},V_{\bar0})=0$.
\item[iii)] $b|_{V_{\bar0}\times V_{\bar0}}$ is the polar of $q_{\bar0}$. That is, $b(x_{\bar0},y_{\bar0})=q_{\bar0}(x_{\bar0}+y_{\bar0})-q_{\bar0}(x_{\bar0}) -q_{\bar0}(y_{\bar0})$ for any $x_{\bar0},y_{\bar0}\in V_{\bar0}$.
\end{enumerate}
The quadratic superform $q=(q_{\bar0},b)$ is said to be {\it regular} if $q_{\bar0}$ is regular and the alternating form $b|_{V_{\bar1}\times V_{\bar1}}$ is nondegenerate ($q_{\bar0}$ regular as in definition in \cite[p. xix]{KMRT}). Note that this implies that $b$ is nondegenerate unless $\chara F=2$ and $\dim V_{\bar0}=1$ (and they are equivalent facts if $\chara F\neq2$).
\end{df}

\begin{df}\label{dfcompsuper} 
A superalgebra $C=C_{\bar0}\oplus C_{\bar1}$ over a field $F$, endowed with a regular quadratic superform $q=(q_{\bar0},b):C\rightarrow F$ (called the {\it norm}), is said to be a {\it composition superalgebra} in case
\begin{enumerate}
\item[i)] $q_{\bar0}(x_{\bar0}y_{\bar0}) = q_{\bar0}(x_{\bar0})q_{\bar0}(y_{\bar0})$
\item[ii)] $b(x_{\bar0}y,x_{\bar0}z) = q_{\bar0}(x_{\bar0})b(y,z)=b(yx_{\bar0},zx_{\bar0})$
\item[iii)] $b(xy,zt)+(-1)^{|x||y|+|x||z|+|y||z|}b(zy,xt) = (-1)^{|y||z|}b(x,z)b(y,t)$
\end{enumerate}
for any $x_{\bar0},y_{\bar0}\in C_{\bar0}$ and homogeneous elements $x,y,z,t\in C$. (Here $|x|$ means the parity of the homogeneous element $x$).

The unital composition superalgebras are called {\it Hurwitz superalgebras} (note that $C_{\bar0}$ is a Hurwitz algebra in this case); also, $\bar x = b(x,1)1 - x$ is called the {\it canonical involution} of the Hurwitz superalgebra. On the other hand, the composition superalgebras satisfying $b(xy,z)=b(x,yz)$ for any $x,y,z$ are called {\it symmetric composition superalgebras}.
\end{df}

\begin{example}
The superalgebra $B(1,2)$ (\cite{Shestakov}):\\
Let $F$ be a field of characteristic $3$ and $V$ a $2$-dimensional vector space over $F$ with an alternating nondegenerate form $(\cdot,\cdot)$.
Consider the superspace $$B(1,2)=F1 \oplus V$$ with $B(1,2)_{\bar 0}=F1, \, B(1,2)_{\bar 1}=V$, and supercommutative multiplication given by $1x=x1=x$ and $uv=(u,v)1$ for all $x\in B(1,2)$ and $u,v\in V$. $B(1,2)$ is a Hurwitz superalgebra (\cite{EldOku02}) with the norm given by
$q_{\bar 0}(1)=1,\, b(1,V)=0$ and $b(u,v)=(u,v)$ for all $u,v\in V$.
\end{example}

\begin{example}
The superalgebra $B(4,2)$ (\cite{Shestakov}):\\
Let $V$ be as above (so $\chara F=3$). Consider ${\rm End}_F(V)$ with the symplectic involution
$\varphi\mapsto{\bar\varphi}:={\rm tr}(\varphi)1-\varphi$ for all $\varphi\in {\rm End}_F(V)$.
Define $$B(4,2)={\rm End}_F(V) \oplus V$$ with $B(4,2)_{\bar 0}={\rm End}_F(V)$ and $B(4,2)_{\bar 1}=V$.
The multiplication is given by
\begin{enumerate}
\item[$\bullet$] the composition of maps in ${\rm End}_F(V)$,
\item[$\bullet$] $v\cdot\varphi=\varphi(v)={\bar\varphi}\cdot v$ for all $\varphi\in {\rm End}_F(V), \, v\in V$,
\item[$\bullet$]$u\cdot v=(-,u)v\in {\rm End}_F(V) \, (w\mapsto(w,u)v)$ for all $u,v\in V$.
\end{enumerate}
$B(4,2)$ is a Hurwitz superalgebra (\cite{EldOku02}) with the norm given by $q_{\bar 0}(\varphi)={\rm det}(\varphi)$ for all $\varphi\in{\rm End}_F(V)$,
$b({\rm End}_F(V),V) = 0 = b(V,{\rm End}_F(V))$ and $b(u,v)=(u,v)$ for all $u,v\in V$.
\end{example}

\medskip

Notice that the superalgebra $\CD(Q,\alpha)$, defined as in Notation \ref{CDdef}, is a Hurwitz superalgebra if $\chara F=2$ (as proved in \cite[Example 2.8]{EldOku02}), with quadratic superform $(q_{\bar0}, b)$, where $q_{\bar0}$ is the restriction of the norm $q$ of $\CD(Q,\alpha)$ to $Q$ and $b$ is the polar form of $q$. Note that $q$ is the orthogonal sum $q_{\bar0}\perp q_{\bar1}$, where $q_{\bar1}(xu)=-\alpha q_{\bar0}(x)$ for all $x\in Q$.

Recall from \cite[Theorem 3.1]{EldOku02} the following classification of the Hurwitz superalgebras:
\begin{theorem}
Let $C$ be a Hurwitz superalgebra over a field $F$. Then either:
\begin{enumerate}
\item[\textnormal{i)}] $C_{\bar1}=0$, or
\item[\textnormal{ii)}] $\chara F=3$ and $C$ is isomorphic either to $B(1,2)$ or to $B(4,2)$, or
\item[\textnormal{iii)}] $\chara F=2$ and $C$ is isomorphic to a superalgebra $\CD(Q,\alpha)$ for a Hurwitz algebra $Q$ of dimension 2 or 4 and a nonzero scalar $\alpha$.
\end{enumerate}
\end{theorem}

\begin{remark} 
It is clear that in case i) the gradings as a superalgebra coincide with the gradings as an algebra on the Hurwitz algebra $C=C_{\bar0}$, which are well known (see \cite{Eld98}). It remains to study the other cases.
\end{remark}

\begin{proposition}
Let $C$ be a Hurwitz superalgebra with quadratic superform $q=(q_{\bar0}, b)$ and a grading $C=\bigoplus_{g\in G} C^g$.
If $g,h\in G$ and $gh\neq e$, then $b(C^g,C^h)=0$. In consequence, if $C^g\neq0$, then $C^{g^{-1}}\neq0$, and the subspaces $C^g$ and $C^{g^{-1}}$ are paired by $b$.
\end{proposition}

\begin{proof}
Assume that $0\neq x\in C^g$, $0\neq y\in C^h$ with $gh\neq e$. We need to prove that $b(x,y)=0$.
Since $C_{\bar0}$ and $C_{\bar1}$ are orthogonal relative to $b$, we can assume that $x,y$ are homogeneous elements and have the same parity.
The result is already known for Hurwitz algebras, so we only have to prove it for odd homogeneous elements $x,y$. In that case, we know that $xy\in C_{\bar0}^{gh}\neq C_{\bar0}^e$, with $C_{\bar0}^{gh}$ and $C_{\bar0}^e$ orthogonal relative to $q_{\bar0}$, so $b(xy,1)=0$. 
From Definition~\ref{dfcompsuper}.iii) we get $b(xy,1)-b(y,x)=b(x,1)b(y,1)=0$, so $0=b(xy,1)=b(y,x)=-b(x,y)$, which proves the first part. Since $b$ is nondegenerate, the second statement is clear.
\end{proof}


\begin{example}\label{ejGradB12} 
Gradings of $B(1,2)$.\\
Let $\Gamma$ be a grading on the superalgebra $B(1,2)$. Take a basis $\{u,v\}$ of homogeneous elements in $V=B(1,2)_{\bar1}$. We can assume that $(u,v)=1$. Then the grading $\Gamma$ is a coarsening of the $\mathbb{Z}$-grading given by:
\begin{equation} \label{eq1}
B(1,2)^{0}=F1, \quad B(1,2)^{1}=Fu, \quad B(1,2)^{-1}=Fv. 
\end{equation}
It is clear that all the gradings of this type (i.e., using different basis of $V$) are equivalent, and the only proper coarsenings are the main grading and the trivial grading. This proves the following:
\end{example}

\begin{theorem}
The nontrivial gradings on $B(1,2)$  are, up to equivalence, the $\mathbb{Z}$-grading (\ref{eq1}) and the main $\mathbb{Z}_2$-grading.
\end{theorem}

\medskip

To give a classification up to isomorphism by an abelian group, we introduce the following notation:

\begin{notation}
Let $G$ be an abelian group (additive notation will be used) and fix a basis $\{u,v\}$ of $V$ with $(u,v)=1$. Given any $g\in G$, denote by $\Gamma(G,g,u,v)$ the $G$-grading induced from the $\mathbb{Z}$-grading in Equation (\ref{eq1}) by the homomorphism $\alpha:\mathbb{Z}\rightarrow G$ determined by $\alpha(1)=g$, that is, $\deg(u)=g$, $\deg(v)=-g$ and $\deg(1)=0$. If $u',v'$ is another basis of $V$ such that $(u',v')=1$, then $\Gamma(G,g,u,v)\rightarrow\Gamma(G,g,u',v')$, $u\mapsto u'$, $v\mapsto v'$, defines an isomorphism of graded superalgebras. Thus, for our purposes we can write $\Gamma(G,g)$ instead of $\Gamma(G,g,u,v)$. It is obvious that $\Gamma(G,g)\cong \Gamma(G,-g)$.
\end{notation}

Now the following result is clear:

\begin{theorem}
Let $\Gamma$ be a $G$-grading on $B(1,2)$ by an abelian group $G$. Then there is an element $g\in G$ such that $\Gamma$ is isomorphic to $\Gamma(G,g)$. Moreover, $\Gamma(G,g)$ is isomorphic to $\Gamma(G,h)$ if and only if either $g=h$ or $g=-h$.
\end{theorem}


\begin{example}\label{ejGradB42} 
Gradings of $B(4,2)$.\\
Let $\Gamma$ be a grading on the superalgebra $B(4,2)$. Take a basis $\{ u, v \}$ of homogeneous elements in $V$ such that $(u,v)=1$.
Using this basis we can identify End$_FV \cong {\mathcal M}_2(F)$ and $V\cong F^2$, and define
$$
e_1= \left( \begin{array}{cc} 1 & 0 \\ 0 & 0 \end{array} \right) , \hspace{.3cm}
e_2= \left( \begin{array}{cc} 0 & 0 \\ 0 & 1 \end{array} \right) , \hspace{.3cm}
x= \left( \begin{array}{cc} 0 & 1 \\ 0 & 0 \end{array} \right) , \hspace{.3cm}
y= \left( \begin{array}{cc} 0 & 0 \\ 1 & 0 \end{array} \right) , \hspace{.3cm}
$$
so we have $u \cdot u = -x, \, v\cdot v = y, \, u\cdot v = -e_2, \, v\cdot u = e_1$. It follows that $e_1, e_2, x, y$ are also homogeneous elements. Besides $e_1\cdot x=x=x\cdot e_2$, so both $e_1$ and $e_2$ are in the neutral component. Thus, $\Gamma$ is a coarsening of the $\mathbb{Z}$-grading (5-grading) given by
\begin{gather} \label{eq2}
\begin{aligned}
B(4,2)^0 =& Fe_1\oplus Fe_2, \quad B(4,2)^1 = Fu, \quad B(4,2)^{-1} = Fv, \\
& B(4,2)^2=Fx, \quad B(4,2)^{-2}=Fy.
\end{aligned}
\end{gather}
(Notice that $\{e_1,e_2,x,y\}$ is a canonical basis (as defined in Section 1) of the split quaternion algebra End$_FV \cong {\mathcal M}_2(F)$.)
Its nontrivial coarsenings are the $\mathbb{Z}_4$-grading:
\begin{equation} \label{eq3} 
B(4,2)^{\bar 0}=Fe_1\oplus Fe_2, \, B(4,2)^{\bar 2}=Fx\oplus Fy, \, B(4,2)^{\bar 1}=Fu, \, B(4,2)^{\bar 3}=Fv, 
\end{equation}
the $\mathbb{Z}_3$-grading:
\begin{equation} \label{eq4}
B(4,2)^{\bar0}=Fe_1\oplus Fe_2, \, B(4,2)^{\bar1}=Fu\oplus Fy, \, B(4,2)^{\bar2}=Fv\oplus Fx,
\end{equation}
and the main $\mathbb{Z}_2$-grading. This gives us all the gradings up to equivalence (over their universal grading groups). Hence we have proved:
\end{example}

\begin{theorem} 
The nontrivial gradings on $B(4,2)$ by their universal grading groups are, up to equivalence, the gradings (\ref{eq2}), (\ref{eq3}), (\ref{eq4}) and the main $\mathbb{Z}_2$-grading.
\end{theorem}

\medskip

Again, we need some more notation to give a classification up to isomorphism.

\begin{notation}
Let $G$ be an abelian group. Take a basis $\{u,v\}$ of $V$ with $(u,v)=1$ and define $e_1,e_2,x,y$ as above. Given $g\in G$, we call $\Gamma(G,g)$ the $G$-grading on $B(4,2)$ induced from the $\mathbb{Z}$-grading in Equation (\ref{eq2}) by the homomorphism $\mathbb{Z}\rightarrow G$, $1\mapsto g$. Notice that $\Gamma(G,g)$ does not depend on the choice of $u$ and $v$, up to isomorphism.
\end{notation}

\begin{theorem}
Let $\Gamma$ be a $G$-grading on $B(4,2)$ by an abelian group $G$. Then $\Gamma$ coincides with $\Gamma(G,g)$ for some $g\in G$ (and some basis $\{u,v\}$ of $V$). Moreover, $\Gamma(G,g)\cong\Gamma(G,h)$ if and only if $g=h$ or $g=-h$.
\end{theorem}

\begin{proof} 
We know that every grading on $B(4,2)$ is, up to equivalence, a coarsening of the $\mathbb{Z}$-grading (\ref{eq2}). Hence, every $G$-grading on $B(4,2)$ coincides with a grading $\Gamma(G,g)$ for some $g\in G$ and some basis $\{u,v\}$ of $V$. If $\Gamma(G,g)\cong\Gamma(G,h)$, the supports of the gradings on the odd component $B(4,2)_{\bar1}$ coincide, i.e., $\{g,-g\}=\{h,-h\}$. Conversely, assume that $\{g,-g\}=\{h,-h\}$. If $g=h$, then $\Gamma(G,g)\cong\Gamma(G,h)$, so we can assume that $h\neq g=-h$. 
Then, we have $\Gamma(G,g)^g_{\bar1}=Fu=\Gamma(G,h)^{-g}_{\bar1}$ and $\Gamma(G,g)^{-g}_{\bar1}=Fv=\Gamma(G,h)^g_{\bar1}$, and the correspondence $u\mapsto v$, $v\mapsto -u$ determines a unique isomorphism $\Gamma(G,g)\rightarrow\Gamma(G,h)$ of graded superalgebras.
\end{proof}


\begin{theorem} \label{hurwitzsuperalggrad} 
Consider the Hurwitz superalgebra $C=\CD(Q,\alpha)$ over a field $F$ with $\chara F=2$, where $Q$ is a Hurwitz algebra of dimension 2 or 4, and $0\neq\alpha\in F$. Let $\Gamma:C=\bigoplus_{g\in G}C^g$ be a fine grading, where $G$ is the universal grading group. Then we have, up to equivalence, one of the following cases:

\noindent $\bullet$ If $\dim Q=2$ ($\dim C=4$), then either:
\begin{enumerate}
\item[\textnormal{i)}] $G=\mathbb{Z}_2$, $Q=C_{\bar0}$ is not split and $\Gamma$ is the main grading, or
\item[\textnormal{ii)}] $G=\mathbb{Z}$ ($3$-grading), $Q=C_{\bar0}$ is split and we have a canonical basis $\{e_1,e_2,u_1,v_1\}$ of $C$ consisting of homogeneous elements such that
    \begin{equation} \label{eq5}
    C_{\bar0}^0=Q=\lspan\{e_1,e_2\}, \ C_{\bar1}^1=\lspan\{u_1\}, \ C_{\bar1}^{-1}=\lspan\{v_1\}.
    \end{equation}
\end{enumerate}

\noindent $\bullet$ If $\dim Q=4$ ($\dim C=8$), we have one of the following cases:
\begin{enumerate}

\item[\textnormal{i)}] $G=\mathbb{Z}_2^2$, $C_{\bar0}^{(\bar0,\bar0)}$ is not split and $\Gamma$ is a refinement of the main $\mathbb{Z}_2$-grading associated to a Cayley-Dickson doubling process as follows: $C=Q\oplus Qu$, $Q=K\oplus Kv$, $K=F1+Fw$ with $u,v,w$ nonisotropic orthogonal elements with zero trace, and
\begin{gather} \label{eq6}
\begin{aligned}
C_{\bar0}^{(\bar0,\bar0)}=K, \; C_{\bar0}^{(\bar1,\bar0)}=Kv, \;  C_{\bar1}^{(\bar0,\bar1)}=Ku, \; C_{\bar1}^{(\bar1,\bar1)}=K(vu).
\end{aligned}
\end{gather}

\item[\textnormal{ii)}] $G=\mathbb{Z}^2$, $Q=C_{\bar0}$ is split, there is a canonical basis of $C$ consisting of homogeneous elements with
$$C_{\bar0}=\lspan\{e_1,e_2,u_3,u_3\}, \ C_{\bar1}=\lspan\{u_1,u_2,v_1,v_2\},$$
and $\Gamma$ is the Cartan $\mathbb{Z}^2$-grading:
\begin{gather} \label{eq7}
\begin{aligned}
C_{\bar0}^{(0,0)} =\lspan & \{e_1,e_2\}, \; C_{\bar0}^{(1,1)} =\lspan\{v_3\}, \; C_{\bar0}^{(-1,-1)}=\lspan\{u_3\}, \\
& C_{\bar1}^{(1,0)} =\lspan\{u_1\}, \; C_{\bar1}^{(-1,0)} =\lspan\{v_1\}, \\
& C_{\bar1}^{(0,1)} =\lspan\{u_2\}, \;  C_{\bar1}^{(0,-1)} =\lspan\{v_2\}.
\end{aligned}
\end{gather}
\end{enumerate}

\end{theorem}

\begin{proof} 
Since $G$ is abelian (see Remark \ref{groupisabelian}), we will use additive notation. We will denote by $q$ the norm of $C$ (and also its polar form) regarded as a Hurwitz algebra. \\
$\bullet$ Case $\dim Q=2$ ($\dim C=4$).\\
Since $\chara F=2$, $\dim Q=2$ and the norm is nondegenerate, it follows that the induced grading in $C_{\bar0}=Q$ is the trivial grading, so $Q = C^0_{\bar0}$. There are two cases:

Assume $G$ is an elementary 2-group. Let $0\neq g\in G$ be such that $C^g_{\bar1}\neq0$. Since $q$ is nondegenerate, we can take $u\in C^g_{\bar1}$ with $q(u)\neq0$. Then, by the Cayley-Dickson doubling process, $C=Q\oplus Qu$ and the grading coincides with the main $\mathbb{Z}_2$-grading. Note that if $Q=C_{\bar0}$ is split, this grading is a coarsening of the grading (\ref{eq5}), so it is not fine.

Assume $G$ is not an elementary 2-group. Then there is $g\in\Supp\Gamma$ with ord$(g)>2$. But $q$ is nondegenerate, so we can take $0\neq x\in C^g$ and $0\neq y\in C^{-g}$ with the same parity and such that $q(x,y)=-1$\,($=1$ as the characteristic is $2$). Since $x$ and $y$ are orthogonal to $1$ and $q(x,y)=-1$, it follows that $e_1:=xy$ and $e_2:=\bar e_1=1-e_1$ are nontrivial idempotents in $C_{\bar0}^0$. Thus, the Peirce spaces relative to the orthogonal idempotents $e_1$ and $e_2$, that is, $U=(e_1C)e_2$ and $V=(e_2C)e_1$, are graded subspaces and we can obtain a canonical basis $\{e_1,e_2,u_1,v_1\}$ of homogeneous elements. Then
$C_{\bar0}^0=Q=\lspan\{e_1,e_2\}$, $C_{\bar1}^g=\lspan\{u_1\}$, $C_{\bar1}^{-g}=\lspan\{v_1\}$, so we have a 3-grading as in Equation~(\ref{eq5}).

\noindent $\bullet$ Case $\dim Q=4$ ($\dim C=8$).

Assume $G$ is an elementary 2-group. If the induced grading on $Q=C_{\bar0}$ is the trivial grading, in the same way as before we see that the grading is the main grading (but this is not a fine grading). Hence, we can take $0\neq g\in G$ such that $C^g_{\bar0}\neq0$. Since $q$ is nondegenerate, there is $v\in C^g_{\bar0}$ with $q(v)\neq0$, and by the Cayley-Dickson doubling process we have the $\mathbb{Z}_2$-graded composition subalgebra $C_{\bar0}^0\oplus C_{\bar0}^0 v$. Since $\chara F=2$, $C_{\bar0}$ does not admit a $\mathbb{Z}_2^2$-grading and therefore $C_{\bar0}=C_{\bar0}^0\oplus C_{\bar0}^0 v$. We repeat the same arguments with a homogeneous element $u\in C_{\bar1}$ such that $q(u)\neq0$, and we obtain $C_{\bar1}=C_{\bar0}u$. Hence $\Gamma$ is a $\mathbb{Z}_2^2$-grading associated to the Cayley-Dickson doubling process (which is a refinement of the main $\mathbb{Z}_2$-grading), as in Equation (\ref{eq6}) .

Assume $G$ is not an elementary 2-group. Then there is $g\in\Supp\Gamma$ with ord$(g)>2$. Since $q$ is nondegenerate we can take elements $0\neq x\in C^g$ and $0\neq y\in C^{-g}$ with the same parity such that $q(x,y)=-1$. Then $e_1:=xy$ and $e_2:=\bar e_1=1-e_1$ are nontrivial idempotents in $C_{\bar0}^0$, and using the Peirce decomposition we obtain a canonical basis $\{e_1,e_2,u_1,u_2,u_3,v_1,v_2,v_3\}$ consisting of homogeneous elements, and where $u_i, v_i$ have the same parity for each $i=1,2,3$. We can assume without loss of generality that $u_3,v_3\in C_{\bar0}$ and $u_1,u_2,v_1,v_2\in C_{\bar1}$. Write $g_i= \deg(u_i)$, so from $u_iv_i=-e_1$ it follows that deg$(v_i)=-g_i$. Since $v_1=u_2u_3$, $g_1+g_2+g_3=0$ and the grading is the Cartan $\mathbb{Z}^2$-grading.

In the case of the $\mathbb{Z}_2^2$-grading (\ref{eq6}), if the subalgebra $C_{\bar0}^{(\bar0,\bar0)}$ were split, then we could take nontrivial idempotents $e_1,e_2=1-e_1\in C_{\bar0}^{(\bar0,\bar0)}$ and complete $\{e_1, e_2\}$ (in the same way as usual) to a canonical basis consisting of homogeneous elements. So in such a case the $\mathbb{Z}_2^2$-grading would not be fine, as it would be a coarsening of the Cartan grading.
\end{proof}


\begin{corollary}\label{corohurwitzsuperalggrad} 
Consider the Hurwitz superalgebra $C=\CD(Q,\alpha)$ over a field $F$ with $\chara F=2$, where $Q$ is a Hurwitz algebra of dimension 2 or 4, and $0\neq\alpha\in F$. Let $\Gamma : C=\bigoplus_{g\in G}C^g$ be a nontrivial grading, where $G$ is the universal grading group. Then:

\noindent $\bullet$ If $\dim Q=2$ ($\dim C=4$), then either
$G=\mathbb{Z}_2$ and $\Gamma$ is the main grading, or
$G=\mathbb{Z}$ and $\Gamma$ is the grading in Equation~\eqref{eq5}.

\noindent $\bullet$ If $\dim Q=4$ ($\dim C=8$), we have up to equivalence one of the following cases:
\begin{enumerate}

\item[\textnormal{i)}] $G=\mathbb{Z}_2^2$ and $\Gamma$ is the grading in Equation~\eqref{eq6}. (Here $K$ is not necessarily split.)

\item[\textnormal{ii)}] $G=\mathbb{Z}_2$ and $\Gamma$ is, either the main $\mathbb{Z}_2$-grading, or the coarsening of the grading (\ref{eq6}) given by:
\begin{equation}\label{cor1eq3} 
C^{\bar0}=K\oplus Ku, \quad C^{\bar1}=Kv\oplus K(vu).
\end{equation}

\item[\textnormal{iii)}] $Q=C_{\bar0}$ is split, there is a canonical basis of $C$ consisting of homogeneous elements such that the main grading is
$$C_{\bar0}=\lspan\{e_1,e_2,u_3,v_3\}, \ C_{\bar1}=\lspan\{u_1,u_2,v_1,v_2\},$$
and we have one of the following cases:
\begin{enumerate}

\item[\textnormal{a)}] $G=\mathbb{Z}^2$ and $\Gamma$ is the Cartan grading given in Equation~\eqref{eq7}.

\item[\textnormal{b)}] $G=\mathbb{Z}$ and $\Gamma$ is a 3-grading, where either
\begin{gather} \label{cor1eq5} \begin{aligned}
& C_{\bar0}^0=Q=\lspan\{e_1,e_2,u_3,v_3\}, \\ 
C_{\bar1}^1= & \lspan\{u_1,v_2\}, \quad C_{\bar1}^{-1}=\lspan\{u_2,v_1\},
\end{aligned} \end{gather}
or
\begin{gather} \label{cor1eq6} \begin{aligned}
& C^0= \lspan\{e_1,e_2,u_1,v_1\}, \\ 
C^1= & \lspan \{u_3,v_2\}, \; C^{-1}=\lspan\{u_2,v_3\}.
\end{aligned} \end{gather}

\item[\textnormal{c)}] $G=\mathbb{Z}$ and $\Gamma$ is a 5-grading, where either
\begin{gather} \label{cor1eq7} \begin{aligned}
C_{\bar0}^0=\lspan & \{e_1,e_2\},  \; C_{\bar0}^2=\lspan\{v_3\}, \; C_{\bar0}^{-2}=\lspan\{u_3\}, \\
& C_{\bar1}^1=\lspan\{u_1,u_2\}, \; C_{\bar1}^{-1}=\lspan\{v_1,v_2\}, 
\end{aligned} \end{gather}
or
\begin{gather} \label{cor1eq8} \begin{aligned}
C_{\bar0}^0=\lspan & \{e_1,e_2\}, \; C_{\bar1}^2=\lspan\{v_1\}, \; C_{\bar1}^{-2}=\lspan\{u_1\}, \\
& C^1=\lspan\{u_2,u_3\}, \; C^{-1}=\lspan\{v_2,v_3\}.
\end{aligned} \end{gather}

\item[\textnormal{d)}] $G=\mathbb{Z}_3$, and
\begin{equation} \label{cor1eq9}
C_{\bar0}^{\bar0}=\lspan\{e_1,e_2\}, \; C^{\bar1}=\lspan\{u_1,u_2,u_3 \}, \; C^{\bar2}=\lspan\{v_1,v_2,v_3\}.
\end{equation}

\item[\textnormal{e)}] $G=\mathbb{Z}_4$, where either
\begin{gather} \label{cor1eq10} \begin{aligned}
& C_{\bar0}^{\bar0}=\lspan\{e_1,e_2\}, \; C_{\bar1}^{\bar1}=\lspan\{u_1,u_2\}, \\
& C_{\bar0}^{\bar2}=\lspan\{u_3,v_3\}, \; C_{\bar1}^{\bar3}=\lspan\{v_1,v_2\},
\end{aligned} \end{gather}
or
\begin{gather} \label{cor1eq11} \begin{aligned}
& C_{\bar0}^{\bar0}=\lspan\{e_1,e_2\}, \;  C^{\bar1}=\lspan\{u_2,u_3\}, \\
& C_{\bar1}^{\bar2}=\lspan\{u_1,v_1\}, \; C^{\bar3}=\lspan\{v_2,v_3\}.
\end{aligned} \end{gather}

\item[\textnormal{f)}] $G=\mathbb{Z}\times\mathbb{Z}_2$, where either
\begin{gather} \label{cor1eq12} \begin{aligned}
C_{\bar0}^{(0,\bar0)}=\lspan\{e_1,e_2\}, \; C_{\bar1}^{(1,\bar0)}=\lspan\{u_2\},
\; C_{\bar1}^{(-1,\bar0)}=\lspan\{v_2\}, \\
C_{\bar1}^{(0,\bar1)}=\lspan\{u_1,v_1\}, \; C_{\bar0}^{(-1,\bar1)}=\lspan\{u_3\},
\; C_{\bar0}^{(1,\bar1)}=\lspan\{v_3\},
\end{aligned} \end{gather}
or
\begin{gather} \label{cor1eq13} \begin{aligned}
C_{\bar0}^{(0,\bar0)}=\lspan\{e_1,e_2\}, \; C_{\bar1}^{(1,\bar0)}=\lspan\{u_2\},
\; C_{\bar1}^{(-1,\bar0)}=\lspan\{v_2\}, \\
C_{\bar0}^{(0,\bar1)}=\lspan\{u_3,v_3\}, \; C_{\bar1}^{(-1,\bar1)}=\lspan\{u_1\},
\; C_{\bar1}^{(1,\bar1)}=\lspan\{v_1\}.
\end{aligned} \end{gather}
\end{enumerate}
\end{enumerate}

\end{corollary}

\begin{proof}
It suffices to compute the coarsenings of the gradings in Theorem~\ref{hurwitzsuperalggrad}.
\end{proof}


\begin{notation}
We will give a classification up to isomorphism of the $G$-gradings on $C$, where $G$ is an abelian group and $C=\CD(Q,\alpha)$ a Hurwitz superalgebra (of dimension 4 or 8) over an algebraically closed field $F$ of characteristic $2$. We will study now the gradings that appear in Corollary~\ref{corohurwitzsuperalggrad}, but before doing that we need to introduce the following notation:

$\bullet$ Consider the case $\dim C=4$. Given any $g\in G$, we will denote by $\Gamma(G,g)$ the $G$-grading induced from the $\mathbb{Z}$-grading given by Equation~\eqref{eq5} by the homomorphism $\mathbb{Z}\rightarrow G$, $1\mapsto g$. Note that the $\mathbb{Z}$-gradings obtained by Equation~\eqref{eq5} for different choices of canonical basis of $Q$ are isomorphic, so that there is no ambiguity in the notation $\Gamma(G,g)$. Since $F$ is algebraically closed, $C$ is split and it is clear that any grading $\Gamma$ on $C$ is a coarsening of a $\mathbb{Z}$-grading given by Equation~\eqref{eq5}, so that there is $g\in G$ such that $\Gamma=\Gamma(G,g)$.

$\bullet$ Consider the case $\dim C=8$. Let $\gamma=(g_1,g_2,g_3)$ be a triple of elements of $G$ with $g_1+g_2+g_3=0$, and take a canonical basis of $C$ such that $C_{\bar0}=\{e_1,e_2,u_3,v_3\}$, $C_{\bar1}=\{u_1,u_2,v_1,v_2\}$. Denote by $\Gamma(G,\gamma)$ the $G$-grading induced from the Cartan $\mathbb{Z}^2$-grading (see Equation~\eqref{eq7}) by the homomorphism $\mathbb{Z}^2\rightarrow G$, $(1,0)\mapsto g_1$, $(0,1)\mapsto g_2$. This is equivalent to set deg$(e_j)=0$ for $j=1,2$, deg$(u_i)=g_i$ and deg$(v_i)=-g_i$ for $i=1,2,3$. Since $F$ is algebraically closed, $C$ is split and by Theorem~\ref{hurwitzsuperalggrad} we get that any grading on $C$ is a coarsening of a $\mathbb{Z}^2$-grading. Hence, if $\Gamma$ is a $G$-grading on $C$, there is $\gamma$ such that $\Gamma=\Gamma(G,\gamma)$. We will say that $\gamma\sim\gamma'$ if there are $\sigma\in\text{Sym(2)}$ and $\varepsilon\in\{\pm1\}$ such that $g_3'=\varepsilon g_3$ and $g_i'=\varepsilon g_{\sigma(i)}$ for $i=1,2$.
\end{notation}

\begin{theorem}
Let $C=\CD(Q,\alpha)$ be a Hurwitz superalgebra over an algebraically closed field $F$ of characteristic 2, and let $G$ be an abelian group.
Let $\Gamma$ be a $G$-grading on $C$. Then, we have up to isomorphism the following classification: if $\dim C=4$ there is some $g\in G$ such that $\Gamma$ is isomorphic to $\Gamma(G,g)$, and if $\dim C=8$ there is some $\gamma=(g_1,g_2,g_3)$ such that $\Gamma$ is isomorphic to $\Gamma(G,\gamma)$. Moreover:
\begin{enumerate}
\item[$\bullet$] $\Gamma(G,g)\cong\Gamma(G,h)$ if and only if $g=h$ or $g=-h$.
\item[$\bullet$] $\Gamma(G,\gamma)\cong\Gamma(G,\gamma')$ if and only if $\gamma\sim\gamma'$.
\end{enumerate}
\end{theorem}

\begin{proof} 
The first part is clear.

If $\Gamma(G,g)\cong\Gamma(G,h)$, the supports of the gradings coincide, so $g=h$ or $g=-h$. Conversely, assume that $g=h$ or $g=-h$. If $g=h$, then $\Gamma(G,g)=\Gamma(G,h)$. Otherwise, $h\neq g=-h$, and we have $\Gamma(G,g)^g=Fu_1=\Gamma(G,h)^h$ and $\Gamma(G,g)^h=Fv_1=\Gamma(G,h)^g$, so there is an isomorphism $\Gamma(G,g)\rightarrow \Gamma(G,h)$ that interchanges $u_1$ and $v_1$.

If $\gamma\sim\gamma'$, it is clear that $\Gamma(G,\gamma)$ and $\Gamma(G,\gamma')$ are isomorphic. Conversely, assume that $\Gamma(G,\gamma)$ and $\Gamma(G,\gamma')$ are isomorphic. Since the even and odd components are invariant by the graded isomorphism, this forces  $g_3'=\varepsilon g_3$ with $\varepsilon \in\{\pm1\}$. In case $C=C^0$, then $g_i=g_i'=0$ for all $i$ and $\gamma\sim\gamma'$. In case $\dim C^0=4$, there is $i\in\{1,2,3\}$ such that $g_i=0$, but using $g_3'=\varepsilon g_3$ and $g_1+g_2+g_3=0$ we see in all cases that $\gamma\sim\gamma'$. Finally, in case $\dim C^0=2$, there are two unique isotropic idempotents in $C^0$, and the isomorphism either fixes them or switches them, that is, it either fixes or switches $U$ and $V$. This means in the first case that $(g_1',g_2',g_3')$ is a permutation of $(g_1,g_2,g_3)$, and in the second case that $(g_1',g_2',g_3')$ is a permutation of $(-g_1,-g_2,-g_3)$. But since $g_3'=\varepsilon g_3$, $\varepsilon=\pm 1$, in both cases $\gamma\sim\gamma'$.
\end{proof}

\section{GRADINGS ON SYMMETRIC COMPOSITION SUPERALGEBRAS}

The goal of this section is to classify the gradings on symmetric composition superalgebras, up to equivalence over any field, and up to isomorphism in the algebraically closed case.

\begin{df} 
Let $(C,\cdot)$ be a Hurwitz superalgebra. Denote by $\bar C$ the vector space $C$ with the new multiplication $x*y:=\bar x\cdot\bar y$. This gives a new composition superalgebra with the same norm and the same main grading. Then $\bar C$ is said to be a {\it para-Hurwitz superalgebra}.
\end{df}

\begin{df}
Let $(C,\cdot)$ be a Hurwitz superalgebra and $\varphi\neq1$ an automorphism of $C$ (as a superalgebra) with $\varphi^3=1$. Denote by $\bar C_\varphi$ the vector space $C$ with the new multiplication $x*y:=\varphi(\bar x)\cdot\varphi^2(\bar y)$. This gives a new composition superalgebra with the same norm and the same main grading. Then $\bar C_\varphi$ is said to be a {\it Petersson superalgebra}. 
\end{df}

Para-Hurwitz superalgebras and Petersson superalgebras are examples of symmetric composition superalgebras (see \cite{EldOku02}).

\begin{example}\label{taus} 
If we take $\{e_1,e_2,u_1,u_2,u_3,v_1,v_2,v_3\}$ a canonical basis of the split Cayley algebra $C$, then we can define the following automorphisms of $C$ of order 3 given by:
\begin{align}
\tau_{st}: &  \ \tau_{st}(e_i)=e_i \ (i=1,2); \\
& \tau_{st}(u_i)=u_{i+1}, \ \tau_{st}(v_i)=v_{i+1} \ \ (\text{subindexes mod 3});  \nonumber \\
\tau_{nst}: & \ \tau_{nst}(e_i)=e_i \ (i=1,2), \\
& \tau_{nst}(u_1)=u_2, \ \tau_{nst}(u_2)=-u_1-u_2, \ \tau_{nst}(u_3)=u_3,  \nonumber \\
& \tau_{nst}(v_1)=-v_1+v_2, \ \tau_{nst}(v_2)=-v_1, \ \tau_{nst}(v_3)=v_3;  \nonumber \\
\tau_\omega: & \ \tau_\omega(e_i)=e_i \ (i=1,2), \ \tau_\omega(u_i)=\omega^i u_i, \ \tau_\omega(v_i)=\omega^{-i}v_i \ (i=1,2,3),
\end{align}
where $\tau_\omega$ is defined if $F$ contains a primitive cubic root $\omega$ of $1$.
Then $P_8(F):=\bar C_{\tau_{st}}$ is the pseudo-octonion algebra over the field $F$, which is isomorphic to $\bar C_{\tau_{nst}}$ and also, if $F$ contains a primitive cubic root of $1$, to $\bar C_{\tau_\omega}$  (see \cite{EldOku02}). The forms of the pseudo-octonion algebra are called Okubo algebras.

Over a field $F$ of characteristic $2$, $C$ is also a Hurwitz superalgebra, where $C_{\bar 0}$ is spanned by $\{e_1,e_2,u_3,v_3\}$ and $C_{\bar 1}$ is spanned by $\{u_1,u_2,v_1,v_2\}$. Then both $\tau_{nst}$ and $\tau_\omega$ (if $\omega\in F$) are automorphisms of $C$ as a superalgebra and hence we may consider the Petersson superalgebras $\bar C_{\tau_{nst}}$ and $\bar C_{\tau_\omega}$.
\end{example}

\begin{example}
Symmetric composition superalgebras $\overline{B(1,2)}_\lambda$ and $\overline{B(4,2)}$.

$\bullet$ Consider the Hurwitz superalgebra $B(1,2)=F1\oplus V$ ($\dim V=2$, $\chara F=3$), and let $(\cdot,\cdot)$ be the alternating form on $V$ which is used to define the product $\cdot$ on $B(1,2)$. Take $\lambda\in F$ and a basis $\{u,v\}$ of $V$ such that $(u,v)=1$. Let $\varphi$ be the automorphism of $B(1,2)$ given by
$$\varphi(u)=u, \, \varphi(v)=\lambda u +v, \, \varphi(1)=1.$$
Then $\varphi^3=1$ and $\varphi$ is an isometry which commutes with the canonical involution.
Define a product $*$ on $B(1,2)$ by $$x*y:=\varphi({\bar x})\cdot\varphi^2({\bar y})$$ for all $x,y\in B(1,2)$, and denote by $\overline{B(1,2)}_\lambda$ the superalgebra $B(1,2)$ with the product $*$.
Then $\overline{B(1,2)}_\lambda$ is a symmetric composition superalgebra. When $\lambda=0$, $\varphi=1$ and $\overline{B(1,2)}_\lambda$  is a para-Hurwitz superalgebra.

$\bullet$ Consider the Hurwitz superalgebra $B(4,2)=\text{End}_F(V) \oplus V$ ($\dim V=2$, $\chara F=3$). We define the product $*$ by $x*y:={\bar x}\cdot{\bar y}$, (as above but this time $\varphi=1$). Denote by $\overline{B(4,2)}$ the superalgebra $B(4,2)$ with the product $*$.
Then $\overline{B(4,2)}$ is a symmetric composition superalgebra, where the main grading is the same as in $B(4,2)$. Moreover, $\overline{B(4,2)}$ is a para-Hurwitz superalgebra.
\end{example}

\bigskip

Recall from \cite[Theorem~4.3]{EldOku02} the classification of the symmetric composition superalgebras:

\begin{theorem}\label{symEldOku}
Let $S$ be a symmetric composition superalgebra over a field $F$. Then either:
\begin{enumerate}
\item[\textnormal{i)}] $S_{\bar 1}=0$,
\item[\textnormal{ii)}] $\chara F=3$ and $S$ is isomorphic to $\overline{B(1,2)}_{\lambda}$ for some $\lambda\in F$,
\item[\textnormal{iii)}] $\chara F=3$ and $S$ is isomorphic to $\overline{B(4,2)}$,
\item[\textnormal{iv)}] $\chara F=2$ and $S$ is isomorphic to the para-Hurwitz superalgebra $\bar Q$, where $Q$ is the Hurwitz superalgebra $Q=\CD(K,\alpha)$, for a $2$-dimensional Hurwitz algebra $K$ and some $\alpha\in F^\times$,
\item[\textnormal{v})] $\chara F=2$ and $S$ is isomorphic to ${\bar C}_\varphi$ where $C=\CD(Q,\alpha)=Q\oplus Qu$ is a Cayley algebra obtained by the Cayley-Dickson doubling process from a quaternion algebra $Q$, with $C_{\bar 0}=Q, \, C_{\bar 1}=Qu=Q^\bot$, and $\varphi$ is an automorphism of $C$ such that $\varphi^3=1$ and
$\varphi|_Q=1$.
\end{enumerate}
\end{theorem}

\begin{df}
The superalgebras given in Theorem \ref{symEldOku}.v) are called \textit{Okubo superalgebras} when $\varphi\neq1$.
\end{df}

\begin{df}
Let $(S,*)$ be a symmetric composition superalgebra. An idempotent element $e\in S_{\bar 0}$ is said to be a {\it para-unit} of $S$ if $e*x=x*e=b(e,x)e-x$ for all $x\in S$.
\end{df}

\begin{lemma}\label{paraunidad} 
Let $(S,*)$ be a symmetric composition superalgebra. If there is a para-unit $e\in S$, then $S$ is a para-Hurwitz superalgebra. Moreover, if $\dim S\geq3$, then the para-unit is unique.
\end{lemma}

\begin{proof} 
Define the multiplication given by $x\cdot y:=(e*x)*(y*e)$. By \cite[Lemma(4.2)]{EldOku02}, we know that $(S,\cdot)$ is a Hurwitz superalgebra, and the multiplication in $(S,*)$ is recovered as $x*y=\varphi(\bar x)\cdot\varphi^2(\bar y)$, where $\varphi:x\mapsto\bar x*e$ (with $\bar x:=b(x,e)e-x$) is an automorphism of $(S,\cdot)$ of order 3. But in this case we have $\varphi=1$, and $(S,*)$ is the para-Hurwitz superalgebra associated to the Hurwitz superalgebra $(S,\cdot)$. 
Moreover, if $\dim S\geq3$, the center of $S$ has dimension 1. But $e$ is an idempotent in the center, so the center of $(S,*)$ is $Fe$ and the para-unit is unique.
\end{proof}

\begin{remark}
In case i) of Theorem~\ref{symEldOku}, it is clear that the gradings as a superalgebra coincide with the gradings as an algebra. But these are well known, because in such a case $S$ is a symmetric composition algebra, and the gradings on symmetric composition algebras were classified in \cite{Eld09}. We will study the remaining cases.
Next lemma covers the para-Hurwitz superalgebras, i.e., it covers the cases ii) with $\lambda=0$, iii), iv) and v) with $\varphi=1$.
\end{remark}

\begin{lemma}
Let $(C,\cdot,q)$ be a Hurwitz superalgebra with $C_{\bar1}\neq0$ over a field $F$ of characteristic 2, and let $(C,*,q)$ be the associated para-Hurwitz superalgebra with $x*y:=\bar x\cdot \bar y$. Then, the gradings on the superalgebras $(C,\cdot,q)$ and $(C,*,q)$ coincide.
Besides, the classifications up to equivalence or up to isomorphism of the gradings in the superalgebras $(C,\cdot,q)$ and $(C,*,q)$ coincide.
\end{lemma}

\begin{proof}
Since $C_{\bar1}\neq0$, $\dim C\geq 3$ (see Theorem~\ref{symEldOku}), and so by Lemma~\ref{paraunidad}, the para-unit of $(C,*)$ is unique (so it coincides with the identity $1$ of $(C,\cdot)$). Note that the center of $(C,*)$ is a graded subspace, and it is generated by the para-unit $1$, so $1$ is always homogeneous of trivial degree ($1$ is an idempotent).

Let $\Gamma : C=\bigoplus_{g\in G}C^g$ be a grading on the superalgebra $(C,\cdot)$, and recall that $\bar x=b(x,1)1-x$.
We know that if $g\neq 0$, then $b(C^g,C^0)=0$, so $\overline{C^g}=C^g$ for all $g\in G$ and then $\Gamma$ is a grading on $(C,*)$. Conversely, if
$\Gamma : C=\bigoplus_{g\in G}C^g$ is a grading on the superalgebra $(C,*)$, this induces a grading on the para-Hurwitz algebra $(C_{\bar0},*)$,
and the unit $1$ of $(C,\cdot)$ is the para-unit of $(C,*)$ which is in $C_{\bar0}^0$, so we have $\bar x=1*x\in C^g$ for all $x\in C^g$ and $g\in G$, and then $\Gamma$ is a grading on the superalgebra $(C,\cdot)$.

Notice that homomorphisms (and isomorphisms) of Hurwitz superalgebras always commute with the canonical involution, so the isomorphisms between two Hurwitz superalgebras coincide with the isomorphisms of their associated para-Hurwitz superalgebras. In particular, an equivalence (resp. isomorphism) between two graded Hurwitz superalgebras is also an equivalence (resp. isomorphism) between their associated graded para-Hurwitz superalgebras, and conversely.
\end{proof}

\begin{remark} 
By the previous lemma, we know that the gradings on the superalgebras $\overline{B(4,2)}$ and $B(4,2)$ coincide, and also the classifications of the gradings up to equivalence and up to isomorphism coincide. By that lemma, this is also true for $\bar Q$ with $Q$ as in Theorem \ref{symEldOku}(iv), and for ${\bar C}_\varphi$ with $C$ as in Theorem \ref{symEldOku}.v) and $\varphi=1$.
\end{remark}

\begin{theorem}
The nontrivial gradings on the superalgebra $S=\overline{B(1,2)}_\lambda$ are:
\begin{enumerate}
\item[$\bullet$] the gradings on $B(1,2)$ (up to equivalence and up to isomorphism), if $\lambda=0$;
\item[$\bullet$] the main $\mathbb{Z}_2$-grading (up to equivalence), if $\lambda\neq0$. In this case, two such gradings are isomorphic if and only if they have the same support.
\end{enumerate}
\end{theorem}

\begin{proof} 
If $\lambda=0$, it is clear by the previous lemma.

In the case $\lambda\neq0$, assume there is a nontrivial grading different from the main grading, and now we will find a contradiction. In this case, we have $S_{\bar 1}=S^g_{\bar 1}\oplus S^h_{\bar 1}$. Since $\lambda\neq0$, $\varphi$ and $\varphi^2$ are not diagonalizable. It is clear that two elements $w_1,w_2\in V$ are
linearly independent if and only if $w_1\cdot w_2\neq0$. Since $\varphi$ is not diagonalizable, we can assume (without loss of generality) that $\varphi(S^g_{\bar 1})\not\subseteq S^g_{\bar 1}$, and so if $0\neq x\in S^g_{\bar 1}$ then $\varphi(x),\varphi^2(x)$ are linearly independent. Therefore $x*x=\varphi({\bar x})\cdot\varphi^2({\bar x})=\varphi(x)\cdot\varphi^2(x)\neq0$ where $x*x\in S^0$, and it follows that $2g=0$. If it were $S^g_{\bar 1}*S^h_{\bar 1}\neq0$, then $g+h=0$ and $h=-g=g$, but this is not possible. Hence $S^g_{\bar 1}*S^h_{\bar 1}=0$, that is,  $\varphi(S^g_{\bar 1})\cdot\varphi^2(S^h_{\bar 1})=0$, so $\varphi(S^g_{\bar 1})=\varphi^2(S^h_{\bar 1})$ and then $S^g_{\bar 1}=\varphi(S^h_{\bar 1})$. In the same way, $S^h_{\bar 1}*S^g_{\bar 1}=0$ and we get that $S^h_{\bar 1}=\varphi(S^g_{\bar 1})$. It follows that $\varphi^2(S^g_{\bar 1})=S^g_{\bar 1}$ and $\varphi^2(S^h_{\bar 1})=S^h_{\bar 1}$,
so $\varphi^2$ would be diagonalizable, which is a contradiction. Thus, the grading is equivalent to the main grading. It is clear that two gradings equivalent to the main grading are isomorphic if and only if they have the same support.
\end{proof}


\begin{remark}\label{remark1section3}
Let $S={\bar C}_\varphi$ be as in Theorem \ref{symEldOku}.v), ($\chara F=2$).\\
By \cite[Remark 4.4]{EldOku02}, we have
\begin{enumerate}[a)]
\item either $\varphi=1$ and $S={\bar C}_\varphi={\bar C}$ is para-Hurwitz, or
\item $\{v\in S_{\bar 1} : \varphi(v)=v\}=0$ and $S={\bar C}_\varphi$ is an Okubo superalgebra.
\end{enumerate}
\end{remark}

It remains to study the case b).

\bigskip

\begin{remark}\label{remark2section3} 
Let $S={\bar C}_\varphi$  with product $*$ and norm $q=(q_{\bar0},b)$, where $\varphi$ is an automorphism of $(C,\cdot)$ that satisfies $\{v\in S_{\bar 1} : \varphi(v)=v\}=0$, $\varphi^3=1$ and $\varphi|_{S_{\bar0}}=1$. Denote by $1$ the para-unit of the para-Hurwitz subalgebra $(S_{\bar0},*)$.

By definition of $S={\bar C}_\varphi$, 
\begin{equation} \label{eq31}
x*y=\varphi(\bar x)\cdot\varphi^2(\bar y)
\end{equation}
for all $x,y\in S$, where $\cdot$ is the product of the Hurwitz superalgebra $(C,\cdot)$. Then $\varphi$ is also an automorphism of $(S,*)$. Besides, for all $x\in S$ we have $$b(1,\varphi(x))=b(\varphi(1),\varphi(x))=b(1,x),$$ and $$\varphi(\bar x)=b(1,x)1-\varphi(x)=b(1,\varphi(x))1-\varphi(x)=\overline{\varphi(x)},$$ that is, $\varphi$ is an isometry that commutes with the involution. 

It follows that
\begin{equation} \label{eq32} 
x\cdot y=(1*x)*(y*1) 
\end{equation} 
for all $x,y\in S$. It is clear that $\varphi(x)=\bar x*1$. Since $\varphi$ commutes with the involution, $\bar x=1*\varphi(x)$.

If $\Gamma:S=\bigoplus_{g\in G}S^g$ is a grading on the superalgebra $(S,*)$, this induces a grading on the para-Hurwitz subalgebra $(S_{\bar0},*)$, so $1\in S_{\bar0}^0$. Hence $\varphi(S^g)=S^g$ and $\overline{S^g}=S^g$ for all $g\in G$ (because $\varphi(x)=1*(1*x)$ and $\bar x = 1*\varphi(x)$). Then, making use of Equation (\ref{eq32}), the grading $\Gamma$ is a grading on the superalgebra $(S,\cdot)$ such that all the subspaces $S^g$ are $\varphi$-invariant. Conversely, let $\Gamma$ be a grading on the superalgebra $(S,\cdot)$ where the subspaces $S^g$ are invariant by $\varphi$. Then, since $\overline{S^g}=S^g$ for all $g\in G$,  Equation~\eqref{eq31} shows that $\Gamma$ is a grading on the superalgebra $(S,*)$.

We conclude that the gradings on the superalgebra $(S,*)$ given in Theorem~\ref{symEldOku}.v) coincide with the gradings on the Cayley superalgebra $(S,\cdot)$ such that all the homogeneous components $S^g$ are $\varphi$-invariant.
\end{remark}

\begin{remark}\label{obse:varphi}
$\textbf{1}.$ Consider the superalgebra $(S,*)$ in Theorem~\ref{symEldOku}.v) in case $\varphi\neq1$ (and recall that $\chara F=2$). By Remark \ref{remark1section3}, we have $\{v\in S_{\bar 1} : \varphi(v)=v\}=0$, $\varphi^3=1$ and $\varphi|_{S_{\bar0}}=1$. If we extend $F$ to an algebraic closure $\bar F$, and given a cubic primitive root of unity $\omega\in\bar F$, then $S_{\bar1}=S(\omega)\oplus S(\omega^2)$ with the subspaces $S(\omega)=\{x\in S_{\bar1}:\varphi(x)=\omega x\}$ and $S(\omega^2)=\{x\in S_{\bar1}:\varphi(x)=\omega^2 x\}$. But $\varphi$ is an isometry, so the restriction of $b$ to both subspaces $S(\omega)$ and $S(\omega^2)$ is zero. Therefore
$S(\omega)$ and $S(\omega^2)$ are paired by $b$ (because $b$ is nondegenerate). In particular, $\dim S(\omega)=2=\dim S(\omega^2)$ and the minimal polynomial
of the restriction of $\varphi$ to the subspace $S_{\bar1}$ is $X^2+X+1$. From now on, we consider again any field $F$.

$\textbf{2}.$ If $(S_{\bar0},\cdot)$ is split, we can take nontrivial idempotents $e_1$, $e_2=\bar e_1=1-e_1$. The subspaces $\widetilde{U}:=U\cap S_{\bar1}=(e_1\cdot S_{\bar1})\cdot e_2$ and $\widetilde{V}:=V\cap S_{\bar1}=(e_2\cdot S_{\bar1})\cdot e_1$ are $\varphi$-invariant and paired by $b$ (because $b$ is nondegenerate), so they are $2$-dimensional. We know that the minimal polynomial of $\varphi|_{\widetilde{U}}$ is either $X+\omega$, or $X+\omega^2$, or $X^2+X+1$. If  it were $X+\omega$, then we could take a basis $\{u_1,u_2\}$ of $\widetilde{U}$ with $0\neq u_1\cdot u_2=v_3\in S_{\bar0}$, so $v_3=\varphi(v_3)=\omega^2 u_1\cdot u_2=\omega^2 v_3$, which is impossible. In the same way, this minimal polynomial cannot be $X+\omega^2$, so it is $X^2+X+1$. Similarly, the minimal polynomial of $\varphi|_{\widetilde{V}}$ is $X^2+X+1$. Thus, we can take a canonical basis of $\widetilde{U}$ such that $\varphi(u_1)=u_2$, $\varphi(u_2)=u_1+u_2$ (recall that $\chara F=2$). Take $u_3\in (e_1\cdot S_{\bar0})\cdot e_2$ such that $b(u_1\cdot u_2,u_3)=1$, and take the dual basis $\{v_1,v_2,v_3\}$ of $\{u_1,u_2,u_3\}$, so we have a canonical basis of $(S,\cdot)$.
The orthogonality of $S_{\bar0}$ and $S_{\bar1}$ relative to $b$ implies that $v_3\in S_{\bar0}$ and $v_1,v_2\in\widetilde{V}$. Furthermore, from $v_1=u_2\cdot u_3$ it follows that $\varphi(v_1)=(u_1+u_2)\cdot u_3=v_1+v_2$, and from $v_2=u_3\cdot u_1$ it follows that $\varphi(v_2)=u_3\cdot u_2=v_1$, so using this canonical basis we can identify $\varphi=\tau_{nst}$. If $\omega\in F$, we could take a basis $\{u_1, u_2\}$ of $\widetilde{U}$ such that $\varphi(u_1)=\omega u_1$ and $\varphi(u_2)=\omega^2 u_2$, and proceed as above to identify $\varphi=\tau_\omega$. This proves the existence, when $(S_{\bar0},\cdot)$ is split, of the gradings associated to a canonical basis of $(S,\cdot)$ in the following result:
\end{remark}


\begin{theorem}\label{theo:gradings_Svarphi} 
Let $\Gamma$ be a grading on the Okubo superalgebra $S=(\bar C_\varphi,*)$ over a field $F$ of characteristic $2$, and let $G$ be its universal grading group. Denote by $\cdot$ the product of the associated Hurwitz superalgebra $C$. Then $\Gamma$ is, up to equivalence, one of the gradings given in the following cases:

\begin{enumerate}
\item[\textnormal{(1)}] $G=\mathbb{Z}_2^2$, the quaternion algebra $(S_{\bar0},\cdot)$ contains a 2-dimensional subalgebra $K=F1\oplus Fw$ with $w^{\cdot2}+w+1=0$, there is $0\neq\alpha\in F$ such that $C=\CD(S_{\bar0},\alpha)=S_{\bar0}\oplus S_{\bar0}\cdot u$, $S_{\bar1}=S_{\bar0}\cdot u$, $u^{\cdot2}=\alpha$, and $\varphi$ is the automorphism of order 3 determined by $\varphi|_{S_{\bar0}}=1$ and $\varphi(u)=w\cdot u$. Moreover, 
\begin{equation} \label{okuboeq1}
S_{\bar0}^{(\bar0,\bar0)}=K, \; S_{\bar0}^{(\bar1,\bar0)}=K^\perp\cap S_{\bar0}, \; 
S_{\bar1}^{(\bar0,\bar1)}=K\cdot u, \; S_{\bar1}^{(\bar1,\bar1)}=(K^\perp\cap S_{\bar0})\cdot u.
\end{equation}

\item[\textnormal{(2)}] $G=\mathbb{Z}_2$ and $\Gamma$ is either the main $\mathbb{Z}_2$-grading, or a coarsening of the grading in Equation~(\ref{okuboeq1}) given by
\begin{equation} \label{okuboeq2}
S^{\bar0}=K\oplus K\cdot u, \; S^{\bar1}=(K^\perp\cap S_{\bar0}) \oplus (K^\perp\cap S_{\bar0})\cdot u. 
\end{equation}

In the cases $(3)$-$(8)$, $S_{\bar0}$ is split, $\tau_{nst}$ and $\tau_\omega$ (as in Example~\ref{taus}) are defined using a canonical basis of $(S,\cdot)$ such that $S_{\bar0}=\lspan\{e_1,e_2,u_3,v_3\}$ and $S_{\bar1}=\lspan\{u_1,u_2,v_1,v_2\}$, and we assume that $\omega\in F$ in case $\varphi=\tau_\omega$.

\item[\textnormal{(3)}] $G=\mathbb{Z}^2$, $\varphi=\tau_\omega$ and $\Gamma$ is given by
\begin{gather} \label{okuboeq3} \begin{aligned}
C_{\bar0}^{(0,0)}=\lspan & \{e_1,e_2\}, \; C_{\bar0}^{(1,1)}=\lspan\{v_3\}, \; C_{\bar0}^{(-1,-1)}=\lspan\{u_3\}, \\
& C_{\bar1}^{(1,0)}=\lspan\{u_1\}, \ C_{\bar1}^{(-1,0)}=\lspan\{v_1\}, \\
& C_{\bar1}^{(0,1)}=\lspan\{u_2\}, \ C_{\bar1}^{(0,-1)}=\lspan\{v_2\}.
\end{aligned} \end{gather}

\item[\textnormal{(4)}] $G=\mathbb{Z}$ (3-grading), $\varphi=\tau_\omega$, and either
\begin{equation} \label{okuboeq4}
C_{\bar0}^0=\lspan\{e_1,e_2,u_3,v_3\}, \ C_{\bar1}^1=\lspan\{u_1,v_2\}, \ C_{\bar1}^{-1}=\lspan\{u_2,v_1\},
\end{equation}
or
\begin{equation} \label{okuboeq5}
C^0=\lspan\{e_1,e_2, u_1,v_1\}, \;
C^1=\lspan\{u_3, v_2\}, \;
C^{-1}=\lspan\{u_2, v_3 \}.
\end{equation}

\item[\textnormal{(5)}] $G=\mathbb{Z}$ (5-grading), and either $\varphi=\tau_{nst}$ with
\begin{gather} \label{okuboeq6} \begin{aligned}
C_{\bar0}^0=\lspan & \{e_1,e_2\}, \; C_{\bar1}^1=\lspan\{u_1,u_2\}, \; C_{\bar0}^2=\lspan\{v_3\}, \\
& C_{\bar1}^{-1}=\lspan\{v_1,v_2\}, \; C_{\bar0}^{-2}=\lspan\{u_3\},
\end{aligned} \end{gather}
or $\varphi=\tau_\omega$ with
\begin{gather} \label{okuboeq7} \begin{aligned}
C_{\bar0}^0=\lspan & \{e_1,e_2\}, \; C^1=\lspan\{u_3, u_2\}, \; C_{\bar1}^2=\lspan\{v_1\}, \\
& C^{-1}=\lspan\{v_3, v_2\}, \; C_{\bar1}^{-2}=\lspan\{u_1\}.
\end{aligned} \end{gather}

\item[\textnormal{(6)}] $G=\mathbb{Z}_3$, $\varphi=\tau_{nst}$ and
\begin{gather} \label{okuboeq8} \begin{aligned}
C_{\bar0}^{\bar0}=\lspan\{e_1,e_2\}, \;
C^{\bar1}=\lspan\{u_1,u_2,u_3\}, \; 
C^{\bar2}=\lspan\{v_1,v_2,v_3\}. 
\end{aligned} \end{gather}

\item[\textnormal{(7)}] $G=\mathbb{Z}_4$, and either $\varphi=\tau_{nst}$ with
\begin{gather} \label{okuboeq9} \begin{aligned}
C_{\bar0}^{\bar0}=\lspan\{e_1,e_2\}, \; C_{\bar1}^{\bar1}=\lspan\{u_1,u_2\}, \\
C_{\bar0}^{\bar2}=\lspan\{u_3,v_3\}, \; C_{\bar1}^{\bar3}=\lspan\{v_1,v_2\},
\end{aligned} \end{gather}
or $\varphi=\tau_\omega$ with
\begin{gather} \label{okuboeq10} \begin{aligned}
C_{\bar0}^{\bar0}=\lspan\{e_1,e_2\}, \; C^{\bar1}=\lspan\{u_2,u_3\}, \\ 
C_{\bar1}^{\bar2}=\lspan\{u_1,v_1\}, \; C^{\bar3}=\lspan\{v_2,v_3\}.
\end{aligned} \end{gather}

\item[\textnormal{(8)}] $G=\mathbb{Z}\times\mathbb{Z}_2$, $\varphi=\tau_\omega$, and either
\begin{gather} \label{okuboeq11} \begin{aligned}
C_{\bar0}^{(0,\bar0)}=\lspan\{e_1,e_2\}, \ C_{\bar1}^{(1,\bar0)}=\lspan\{u_2\},
\ C_{\bar1}^{(-1,\bar0)}=\lspan\{v_2\}, \\
C_{\bar1}^{(0,\bar1)}=\lspan\{u_1,v_1\}, \ C_{\bar0}^{(-1,\bar1)}=\lspan\{u_3\},
\ C_{\bar0}^{(1,\bar1)}=\lspan\{v_3\},
\end{aligned} \end{gather}
or
\begin{gather} \label{okuboeq12} \begin{aligned}
C_{\bar0}^{(0,\bar0)}=\lspan\{e_1,e_2\}, \; C_{\bar1}^{(1,\bar0)}=\lspan\{u_2\},
\; C_{\bar1}^{(-1,\bar0)}=\lspan\{v_2\}, \\
C_{\bar0}^{(0,\bar1)}=\lspan\{u_3,v_3\}, \; C_{\bar1}^{(-1,\bar1)}=\lspan\{u_1\},
\; C_{\bar1}^{(1,\bar1)}=\lspan\{v_1\}. 
\end{aligned} \end{gather}

\end{enumerate}

\end{theorem}

\begin{proof}
By the conclusion of Remark \ref{remark2section3}, we have to find the gradings on the superalgebra $(S,\cdot)$ where the homogeneous components are $\varphi$-invariant. We do this by checking all the cases in Corollary \ref{corohurwitzsuperalggrad}. Denote by $q$ the norm of $(S,\cdot)$ regarded as a Hurwitz algebra.

$\textbf{(1), (2).}$ Assume we have a $\mathbb{Z}_2^i$-grading ($i=1,2$).
Let $u\in S_{\bar1}$ be a homogeneous element with $q(u)=\alpha\neq0$. Then $S_{\bar1}=S_{\bar0}\cdot u$. But $\varphi(S_{\bar1})=S_{\bar1}$, so $\varphi(u)=w\cdot u$ for some $w\in S_{\bar0}$, and it is clear that $w$ is a homogeneous element. Since $u\cdot u=\alpha\in F$, it follows that $w\notin F$ (because otherwise we would have $q(u)=q(\varphi(u))=q(w\cdot u)=w^{\cdot 2}q(u)$, that implies $w^{\cdot2}=1$ because $\varphi$ is an isometry, but also $w^{\cdot3}=1$ because $\varphi^3=1$, so $w=1$, which contradicts $\varphi\neq1$). The minimal polynomial of the restriction $\varphi|_{S_{\bar1}}$ is $X^2+X+1$, so it follows that $w^{\cdot2}+w+1=0$ and $w\in S_{\bar0}^0$. Then we have either the $\mathbb{Z}^2_2$-grading in Equation~\eqref{okuboeq1}, or a coarsening of this grading (which is either the $\mathbb{Z}_2$-grading in Equation~\eqref{okuboeq2}, or the main grading).

This proves cases $(1)$ and $(2)$ of the theorem.

\smallskip

It remains to prove that there are no more gradings associated to canonical basis (uniqueness), up to equivalence.
Thus, it is enough to prove that for each grading on the superalgebra $(S,\cdot)$, the automorphisms $\varphi$ of $(S,\cdot)$ that fix the homogeneous components are exactly the ones listed in the Theorem. To have any of these gradings, it is necessary that $(S_{\bar0},\cdot)$ be split, so in all these cases we will have a canonical basis of homogeneous elements such that $S_{\bar0}=\lspan\{e_1,e_2,u_3,v_3\}$ and $S_{\bar1}=\lspan\{u_1,u_2,v_1,v_2\}$ (see Section 1). Also, recall from Theorem~\ref{symEldOku}.v) that $\varphi|_{S_{\bar0}}=1$.

$\textbf{(3).}$ Assume that we have the $\mathbb{Z}^2$-grading on the superalgebra $(C,\cdot)$ given in Equation~\eqref{eq7}.
In this case we have 1-dimensional $\varphi$-invariant subspaces in the odd component, so they are associated to the eigenvalues $\omega$ and $\omega^2$, and $\omega\in F$. Then $\varphi(u_1)=\omega^i u_1$ with $i\in\{1,2\}$. From $u_1\cdot u_2=v_3$, $u_1\cdot v_1=e_1$ and $v_1\cdot v_2=u_3$ it follows that $\varphi(u_2)=\omega^{2i} u_2$, $\varphi(v_1)=\omega^{2i} v_1$ and $\varphi(v_2)=\omega^i v_2$. In both cases $i=1,2$, the gradings are equivalent, where the equivalence of graded algebras is given by $e_1\mapsto e_2$, $e_2\mapsto e_1$, $u_i\mapsto v_i$ and $v_i\mapsto u_i$.

$\textbf{(4).}$ Assume that we have the $\mathbb{Z}$-grading on the superalgebra $(C,\cdot)$ given in Equation~\eqref{cor1eq5}.
Note that $(e_1\cdot\lspan\{u_1,v_2\})\cdot e_2=Fu_1$ is $\varphi$-invariant, and the same happens with $Fu_2$, $Fv_1$ and $Fv_2$, so $\omega\in F$. As before, for some $i=1,2$, we have $\varphi(u_1)=\omega^i u_1$, $\varphi(v_1)=\omega^{2i} v_1$, $\varphi(u_2)=\omega^{2i} u_2$, $\varphi(v_2)=\omega^i v_2$, but both cases $i=1,2$ are equivalent, as in case (3).

Assume now that we have the grading in Equation~\eqref{cor1eq6}. Then $\varphi$ preserves the subspaces $(e_1\cdot\lspan\{v_3,u_2\})\cdot e_2=Fu_2$, and $(e_2\cdot\lspan\{u_3,v_2\})\cdot e_1=Fv_2$. From $u_2\cdot v_2=-e_1=e_1$ (recall that $\chara F=2$), it follows that $\varphi(u_2)=\omega^i u_2$ and $\varphi(v_2)=\omega^{2i} v_2$ for some $i\in \{1,2\}$; from $u_1=v_2\cdot v_3$ and $v_1=u_2\cdot u_3$, we have also $\varphi(u_1)=\omega^{2i} u_1$ and $\varphi(v_1)=\omega^i v_1$. Both cases $i=1,2$ are equivalent, with the same equivalence of graded algebras given in case (3).

$\textbf{(5).}$ Assume that we have the $\mathbb{Z}$-grading on the superalgebra $(C,\cdot)$ given in Equation~\eqref{cor1eq7}.
We have the $\varphi$-invariant subspaces $\lspan\{u_1,u_2\}$ and $\lspan\{v_1,v_2\}$, but $u_1\cdot u_2=v_3$ and $v_1\cdot v_2=u_3$, so it follows (as in Remark~\ref{obse:varphi}) that the restrictions of $\varphi$ to both subspaces have minimal polynomial $X^2+X+1=0$, and the elements $u_1,u_2,v_1,v_2$ can be changed by other elements in the same subspaces to identify $\varphi=\tau_{nst}$. For the case of Equation~\eqref{cor1eq8} the proof is similar.

$\textbf{(6), (7), (8).}$ Assume that we have a $G$-grading on the superalgebra $(C,\cdot)$ given in Corollary \ref{corohurwitzsuperalggrad} with $G=\mathbb{Z}_3$, $\mathbb{Z}_4$ or $\mathbb{Z}\times\mathbb{Z}_2$. These cases are proved using the same previous arguments.

This proves cases $(3)$-$(8)$ of the Theorem.
\end{proof}

\begin{remark}
There are always gradings like the ones described in the cases (1) and (2) of Theorem~\ref{theo:gradings_Svarphi}. On the other hand, the cases with $\varphi=\tau_{nst}$ exist if and only if $(S_{\bar0},\cdot)$ is split, and the cases with $\varphi=\tau_\omega$ exist if and only if $(S_{\bar0},\cdot)$ is split and $\omega\in F$ (see Remark~\ref{obse:varphi}.2). 
\end{remark}


Our last goal is to classify the gradings up to isomorphism on Okubo superalgebras over an algebraically closed field $F$.

Let $G$ be an abelian group and let $C=\CD(Q,\alpha)$ be the split Cayley superalgebra over an algebraically closed field $F$ ($\chara F=2$), and denote by $\cdot$ its product. Fix an automorphism $\varphi\neq1$ of order 3 of the superalgebra $C$ with $\varphi|_Q=1$. As in Remark~\ref{obse:varphi}.2, from now on we may assume $\varphi=\tau_\omega$. Let $S=\bar C_\varphi$ be the Okubo superalgebra associated to $\varphi$, and denote its product by $*$. To give a classification up to isomorphism of the $G$-gradings on $S$, we introduce the following notation:

\begin{notation} 
Let $\gamma=(g_1,g_2,g_3)$ be a triple of elements of $G$ with $g_1+g_2+g_3=0$, and take a canonical basis of $C$ such that $C_{\bar0}=\lspan\{e_1,e_2,u_3,v_3\}$, $C_{\bar1}=\lspan\{u_1,u_2,v_1,v_2\}$. Denote by $\Gamma(G,\gamma)$ the $G$-grading on $(S,*)$ induced from the $\mathbb{Z}^2$-grading in Equation~\eqref{okuboeq3} by the homomorphism $\mathbb{Z}^2\rightarrow G$, $(1,0)\mapsto g_1$, $(0,1)\mapsto g_2$. This is equivalent to set deg$(e_j)=0$ for $j=1,2$, deg$(u_i)=g_i$ and deg$(v_i)=-g_i$ for $i=1,2,3$. Since $F$ is algebraically closed, Theorem~\ref{theo:gradings_Svarphi} shows that any grading on $(S,*)$ is, up to isomorphism, a coarsening of the $\mathbb{Z}^2$-grading. Actually, for any grading on $(S,*)$, since $F$ is algebraically closed, the associated Hurwitz algebra $(S_{\bar 0}^0,\cdot)$ contains two orthogonal idempotents $e_1$ and $e_2=1-e_1$, which can be completed to a canonical basis (see Remark~\ref{obse:varphi}) consisting of homogeneous elements so that $\varphi=\tau_\omega$.
Hence, if $\Gamma$ is a $G$-grading on $(S,*)$, there is $\gamma$ such that $\Gamma=\Gamma(G,\gamma)$. We will say that $\gamma\sim\gamma'$ if there are $\sigma\in\text{Sym(2)}$ and $\varepsilon \in\{\pm1\}$ such that $g_3'=\varepsilon g_3$ and $g_i'=\varepsilon g_{\sigma(i)}$ for $i=1,2$.
\end{notation}

\begin{theorem}
Let $C=\CD(Q,\alpha)$ be the eight-dimensional Hurwitz superalgebra over an algebraically closed field $F$ of characteristic $2$ and let $G$ be an abelian group. Fix an automorphism $\varphi\neq1$ of order 3 of the superalgebra $C$, and let $S=\bar C_\varphi$ be the Okubo superalgebra associated to $\varphi$. Let $\Gamma$ be a $G$-grading on $S$. Then, there is some $\gamma=(g_1,g_2,g_3)$ such that $\Gamma$ is isomorphic to $\Gamma(G,\gamma)$. Moreover,
$\Gamma(G,\gamma)\cong\Gamma(G,\gamma')$ if and only if $\gamma\sim\gamma'$.
\end{theorem}

\begin{proof}
The first assertion is clear from the comments above.

If $\gamma\sim\gamma'$, it is clear that  $\Gamma(G,\gamma)$ and $\Gamma(G,\gamma')$ are isomorphic. Conversely, assume that $f$ is an automorphism of the superalgebra $(S,*)$ sending $\Gamma(G,\gamma)$ to $\Gamma(G,\gamma')$. As the even and odd components are invariant by the graded isomorphism $f$, this forces $g_3'=\varepsilon g_3$ with $\varepsilon \in\{\pm1\}$. In case $C=C^0$, $g_i=g_i'=0$ for all $i$ and $\gamma\sim\gamma'$. In case $\dim C^0=4$, there is $i\in\{1,2,3\}$ such that $g_i=0$, but using $g_3'=\varepsilon g_3$ and $g_1+g_2+g_3=0$ we see in all cases that $\gamma\sim\gamma'$. Finally, in case $\dim C^0=2$, there are two unique idempotents in $(C^0,\cdot)$, so the graded automorphism $f$ of $(S,*)$ either fixes them or switches them, that is, it either fixes or switches $U$ and $V$ (associated to the Cayley algebra $(C,\cdot)$). This means in the first case that $(g_1',g_2',g_3')$ is a permutation of $(g_1,g_2,g_3)$, and in the second case that $(g_1',g_2',g_3')$ is a permutation of $(g_1^{-1},g_2^{-1},g_3^{-1})$. But we had $g_3'=\varepsilon g_3$, so in both cases $\gamma\sim\gamma'$.
\end{proof}



\begin{thebibliography}{9}
\bibliographystyle{alpha}

\bibitem[Eld98]{Eld98}
A. Elduque, \textit{Gradings on octonions}, J. Algebra \textbf{207} (1998), pp. 342--354.

\bibitem[Eld09]{Eld09}
A. Elduque, \textit{Gradings on symmetric composition algebras}, J. Algebra \textbf{322} (2009), no. 10, pp. 3542--3579.

\bibitem[EK12]{EldKoc12}
A. Elduque and M. Kochetov, \textit{Gradings on the exceptional Lie algebras F4 and G2
revisited}, Rev. Mat. Iberoam. \textbf{28} (2012), no. 3, pp. 773--813.

\bibitem[EK13]{EldKoc13}
A. Elduque and M. Kochetov, \textit{Gradings on Simple Lie Algebras}, Mathematical Surveys and Monographs \textbf{189}, American Mathematical Society, 2013.

\bibitem[EO02]{EldOku02}
A. Elduque and S. Okubo, \textit{Composition Superalgebras}, Comm. Algebra \textbf{30} (2002), no. 11,
pp. 5447--5471.

\bibitem[KMRT98]{KMRT}
M.A. Knus, A.S. Merkjurev, M. Rost, J.P. Tignol, {\it The Book of Involutions}, American Mathematical Society, Colloquium Publications \textbf{44}, Providence, 1998.

\bibitem[She97]{Shestakov}
I.P.~Shestakov, \textit{Prime alternative superalgebras of arbitrary characteristic}, Algebra and Logic \textbf{36} (1997), no.~6, pp. 389--412.

\bibitem[ZSSS82]{ZSSS}
K.A. Zhevlakov, A.M. Slin'ko, I.P. Shestakov, A.I. Shirshov, {\it Rings that are nearly associative}, Academic Press, 1982.


\end{thebibliography}
\end{document}